\documentclass[a4paper, 12pt]{article}

\usepackage{mathrsfs}
\usepackage{epsfig}
\usepackage{amsmath}
\usepackage{amssymb}
\usepackage{latexsym}
\usepackage{amsfonts}
\usepackage{amsthm}

\usepackage[numbers,sort&compress]{natbib}
\usepackage{graphicx}
\ifx\pdfoutput\undefined  \else
 \fi

\usepackage[centerlast]{caption2}

\usepackage{color}

\usepackage{txfonts}
\usepackage{amssymb}
\usepackage{mathrsfs}
\usepackage{amsfonts}

\newtheorem{theorem}{Theorem}[section]
\newtheorem{lemma}[theorem]{Lemma}
\newtheorem{cor}[theorem]{Corollary}

\newtheorem{conj}[theorem]{Conjecture}

\usepackage{latexsym,amssymb}
\usepackage{amsmath}

\begin{document}

\title{ On the Erd\H{o}s-Ginzburg-Ziv invariant and zero-sum Ramsey number for intersecting families\footnote{The research is
supported by NSFC (11001035, 11271207) and Science and Technology Development Fund of Tianjin Higher Institutions (20121003).}}
\author{
Haiyan Zhang$^a$ \ \ \ \ \ \ Guoqing Wang$^b$\thanks{Corresponding author's E-mail :
gqwang1979@yahoo.com.cn}\\
{\small $^a$Department of Mathematics,} \\
{\small Harbin University of Science and Technology, Harbin, 150080, P. R. China}\\
{\small $^b$Department of Mathematics,} \\
{\small Tianjin Polytechnic University, Tianjin, 300387, P. R. China}\\
}
\date{}
\maketitle

\begin{abstract} Let $G$ be a finite abelian group, and let $m>0$ with $\exp(G)\mid m$.
Let $s_{m}(G)$ be the generalized Erd\H{o}s-Ginzburg-Ziv invariant which denotes the  smallest positive integer $d$ such that any sequence of elements in $G$ of length $d$ contains a subsequence of length $m$ with sum zero in $G$. For any integer $r>0$, let $\mathcal{I}_m^{(r)}$ be the collection of all $r$-uniform intersecting families of size $m$. Let $R(\mathcal{I}_m^{(r)},G)$ be the smallest positive integer $d$ such that any $G$-coloring of the edges of the complete $r$-uniform hypergraph  $K_{d}^{(r)}$ yields a zero-sum copy of some intersecting family in $\mathcal{I}_m^{(r)}$. Among other results, we mainly prove that  $\Omega(s_{m}(G))-1\leq R (\mathcal{I}_{m}^{(r)}, \  G)\leq \Omega(s_{m}(G)),$ where $\Omega(s_{m}(G))$ denotes the least positive integer $n$ such that  ${n-1 \choose  r-1}\geq s_{m}(G)$, and we show that if $r\mid \Omega(s_{m}(G))-1$ then $R (\mathcal{I}_{m}^{(r)}, \  G)= \Omega(s_{m}(G))$.
\end{abstract}


\section{Introduction}

Erd\H{o}s, Ginzburg and Ziv \cite{EGZ} in 1961 proved the following famous result which is called Erd\H{o}s-Ginzburg-Ziv theorem later.

\noindent\textbf{Throrem A.} \ (Erd\H{o}s-Ginzburg-Ziv). \ {\sl Let $a_1,a_2,\ldots,a_{2n-1}$ be a sequence of $2n-1$ residues modulo $n$. Then there exist $n$ indices $1\leq i_1< i_2<\cdots<i_n\leq 2n-1$ such that
$$a_{i_1}+a_{i_2}+\cdots+a_{i_n}\equiv 0\pmod n.$$ }
Shortly after this, H. Davenport \cite{Davenport} in 1966 posed the problem to determine the smallest positive integer $d$ for any finite abelian group $G$, which is called the Davenport constant $D(G)$, such that every sequence of elements in $G$ of length at least $d$ contains a nonempty subsequence with sum zero in $G$. The Erd\H{o}s-Ginzburg-Ziv theorem and  Davenport constant are the starting point for much research later, which  has been developed into a branch, called zero-sum theory (see \cite{GaoGeroldingersurvey} for a survey), in Combinatorial Number Theory. W.D. Gao \cite{Gaofundamental96} in 1996 find a connection between the Erd\H{o}s-Ginzburg-Ziv theorem and Davenport constant.

\noindent\textbf{Throrem B.} \  (Gao). \ {\sl Let $G$ be a finite abelian group, and let $S$ be a sequence of elements in $G$ of length $|G|+D(G)-1$.  Then $S$ contains a subsequence of length $|G|$ with sum zero.}

 Moreover, Gao \cite{Gao03} go a further step by introducing the zero-sum invariant  $s_{k \exp(G)}(G)$ for any finite abelian group $G$.
For any integer $k>0$, let
$s_{k \exp(G)}(G)$ be the smallest positive integer $d$ such that any sequence $S$ of elements in $G$ of length $d$ contains a subsequence of length $k  \exp(G)$ of sum zero, where $\exp(G)$ is the exponent of $G$. For $k=1$, the invariant $s_{\exp(G)}(G)$ (abbreviated to $s(G)$) is called the Erd\H{o}s-Ginzburg-Ziv invariant. The generalized Erd\H{o}s-Ginzburg-Ziv invariant $s_{k \exp(G)}(G)$  has been studied extensively recently (see \cite{AdhikariRath03,Gao03,GaoHan,GaoThangadurai06,Kubertin}). For the special case $k=1$, the Erd\H{o}s-Ginzburg-Ziv invariant $s(G)$ was studied in a huge of papers (see \cite{GaoGeroldingersurvey,Geroldinger06,GeroldingerRuzsa} for a survey).

Almost meanwhile, some people made a study of zero-sum problems connecting with Ramsey theory. A. Bialostocki and P. Dierker \cite{BialostockiDierker92} in 1992 raised the following interesting variant of the classical Ramsey Theorem: Let $H$ be a graph with $m$ edges and let $t\geq 2$ be an integer with $t\mid m$, and let $\mathbb{Z}_t$ be the cyclic group of order $t$. Define $R(H,\mathbb{Z}_t)$ to be the smallest positive integer $d$ such that for every $\mathbb{Z}_t$-coloring of the edges of the complete graph $K_d$, i.e., a function $c:E(K_d)\rightarrow \mathbb{Z}_t$, there exists in $K_d$ a copy of $H$ such that
$$\sum\limits_{e\in E(H)}c(e)\equiv 0\pmod t.$$ In the same paper, among other results Bialostocki and Dierker determined the precise value of $R(K_{1,m}, \mathbb{Z}_m)$ for star $K_{1,m}$. Very soon after this, Y. Caro \cite{Carostar92} determined the exact value for the zero-sum Ramsey number $R(K_{1,m}, \mathbb{Z}_t)$  for all $t\mid m$, which is stated as follows.

\noindent\textbf{Throrem C.} \ (Bialostocki-Dierker-Caro). \ {\sl Let $m\geq t\geq 2$ be positive integers with $t\mid m$. Then
$$\begin{array}{llll}R (K_{1,m}, \  \mathbb{Z}_t)=\left\{\begin{array}{llll}
               m+t-1,  & \mbox{if \ \ } {m\equiv t\equiv 0 \pmod 2;}\\
              m+t,  & \mbox{if \ \ } otherwise.\\
              \end{array}
              \right.
\end{array}$$
It is worth mentioning that Bialostocki and Dierker also made a study of the zero-sum Ramsey number for matching in hypergraphs (see \cite{BialostockiDierker92}). To generalize the previous results,
Caro \cite{Caro91} and together with Provstgaard \cite{CaroProvstgaard99} studied the zero-sum Ramsey number concerned with a more general combinatorial objects called delta-system in hypergraph setting.
A collection of $r$-sets $e_1,e_2,\ldots,e_m$ is called a {\bf delta-system} provided that there exists a set $Q$ of cardinality $q$, $q<r$, such that $e_i\cap e_j=Q$ for any $1\leq i<j\leq m$. The delta-system $e_1,e_2,\ldots,e_m$ is said to be of type {\bf $S(r,q,m)$}. For $r>q\geq 0$ and $m\geq t\geq 2$ with $t\mid m$, define $R(S(r,q,m), \mathbb{Z}_t)$ to be the smallest positive integer $d$ such that for every $\mathbb{Z}_t$-coloring of the edges of the complete r-uniform hypergraph $K_d^{(r)}$ of $d$ vertices, i.e., a function $c:E(K_d^{(r)})\rightarrow \mathbb{Z}_t$, there exists a delta-system, say $e_1,e_2,\ldots, e_m$, of type $S(r,q,m)$ with each $e_i\in E(K_d^{(r)})$ and
$$\sum\limits_{i=1}^m c(e_i)\equiv 0\pmod t.$$
In \cite{Caro91,CaroProvstgaard99}, Caro and Provstgaard proved the following theorem on zero-sum Ramsey number for delta-systems.

\noindent\textbf{Throrem D.} \ (Caro-Provstgaard). \ {\sl Let $r>q\geq 0$ and $m\geq t\geq 2$ be given integers with $t\mid m$. Then
$$(r-q)m+\max(q,t-1)\leq R(S(r,q,m), \mathbb{Z}_t)\leq (r-q)m+t+q-1.$$
Besides the above Theorem C and Theorem D, more researches (see \cite{AlonCaro93,BialostockiDierker94,CaroYuster94} etc.) were done on zero-sum Ramsey number in the graph or hypergraph setting.

However, all of researches done are concerned with only the $\mathbb{Z}_t$-coloring for some cyclic group $\mathbb{Z}_t$. In this paper, we shall try to generalize the previous results by considering the zero-sum Ramsey number with a
general finite abelian group coloring.

The notations and terminologies on hypergraphs used in this paper will be consistent with  \cite{Bergebook}. For convenience, we give some necessary ones. Let $V=\{v_1,v_2,\ldots,v_n\}$ be a finite set. A {\sl hypergraph} on $V$ is a family $E=(e_1,e_2,\ldots,e_m)$ of subsets of $V$ such that $$e_i\neq \emptyset  \ \ (i=1,2,\ldots,m)$$ and $$\cup_{i=1}^m e_i=V.$$ We denote by $\mathcal{H}=(V; E)$ the hypergraph with vertex set $V=V(\mathcal{H})$ and edge set $E=E(\mathcal{H})$. We call $n=|V|$ and $m=|E|$  the order and the size of the hypergraph $\mathcal{H}$, respectively. For a set $J\subseteq \{1,2,\ldots,m\}$, we call the family $\mathcal{H}'=(e_j: j\in J)$ the partial hypergraph generated by the set $J$. We say the hypergraph $\mathcal{H}$ is $r$-uniform provided that all the edges of $\mathcal{H}$ have cardinality $r$, i.e., $|e_1|=|e_2|=\cdots=|e_m|=r$.   Given the hypergraph $\mathcal{H}$,  we define an {\bf intersecting family} to be a set of edges having nonempty pairwise intersection.
A {\sl hypermatching} in $\mathcal{H}$ is a family of pairwise disjoint edges. The hypergraph $\mathcal{H} = (V; E)$ is said to be a
{\bf hyperstar} if $\cap_{i=1}^{m} e_i\neq \emptyset$.  For any vertex $v\in V$, define $$\mathcal{H}(v)=(e_j: e_j\in E, v\in e_j)$$ to be a maximal hyperstar in $\mathcal{H}$ with $v$ contained in $\bigcap\limits_{e\in E(\mathcal{H}(v))} e$. For a vertex $v\in V$, we define the degree  $d_{\mathcal{H}}(v)$ of $v$ to be the number of edges of $\mathcal{H}(v)$.

\noindent $\bullet$ {\sl In what follows, we shall always denote by $G$ a finite abelian group and by $\exp(G)$ the exponent of $G$.}

 Let $S=(g_1,g_2,\ldots,g_m)$ be a sequence of elements in $G$. We call $S$ a {\sl zero-sum} sequence provide that the sum $g_1+g_2+\cdots+g_m$ of all elements in $S$ equals the identity element $0_G$ of $G$. We call $S$ a {\sl zero-sum free} sequence if $S$ contains no nonempty zero-sum subsequence.
Let $\mathcal{F}=\{\mathcal{H}_i: i\in I\}$ be a family of $r$-uniform hypergraphs such that there exists at least an index $i\in I$ with $\exp(G)\mid m(\mathcal{H}_i)$. We define $R(\mathcal{F}, G)$ to be the smallest positive integer $d$ such that for every $G$-coloring of the edges of $K_{d}^{(r)}$ there exists in $K_d^{(r)}$ a copy of some $\mathcal{H}\in \mathcal{F}$ with $$\sum\limits_{e\in E(\mathcal{H})}c(e)=0_G.$$
If $\mathcal{F}=\{\mathcal{H}\}$ is singleton, we shall write $R(\mathcal{H}, G)$ for $R(\mathcal{F}, G)$. We remark that the classical multicolor Ramsey number ensures the existence of $R(\mathcal{F}, G)$ since $\exp(G)\mid m(\mathcal{H})$ for some $\mathcal{H}\in \mathcal{F}$.

In this paper, we shall make a start on studying the zero-sum Ramsey number for $r$-uniform intersecting families with $G$-coloring.  Let $\mathcal{I}_m^{(r)}$ ($\mathcal{S}_m^{(r)}$) be the collection of all $r$-uniform intersecting families (hyperstars, respectively) of size $m$.
The intersecting family and the hyperstar are combinatorial objects that are fundamental for  Hypergraph Theory and have been studied extensively (see \cite{EKR-AhlKha,Frankl78, Talbot03,EKR-HiltonMilner} for example), all of which originate from the celebrated  Erd\H{o}s-Ko-Rado Theorem \cite{EKR} in 1961. Notice
that both intersecting family and hyperstar are more general than the delta-system,
where it is required the edges must have pairwise the same intersection.

In conclusion, the main result of this paper is Theorem \ref{Theorem Main}.

\begin{theorem}\label{Theorem Main} Let $G$ be a finite abelian group, and let $k\geq 1$, $r\geq 2$ be integers. Then $$\Omega(s_{k\exp(G)}(G))-1\leq R (\mathcal{I}_{k\exp(G)}^{(r)}, \  G)\leq R (\mathcal{S}_{k\exp(G)}^{(r)}, \  G)\leq \Omega(s_{k\exp(G)}(G)),$$ where $\Omega(s_{k\exp(G)}(G))$ denotes the least positive integer $n$ such that  ${n-1 \choose  r-1}\geq s_{k\exp(G)}(G)$. In particular, if $$r\mid \Omega(s_{k\exp(G)}(G))-1,$$ then $R (\mathcal{I}_{k\exp(G)}^{(r)}, \  G)=R(\mathcal{S}_{k\exp(G)}^{(r)}, \  G)= \Omega(s_{k\exp(G)}(G))$.
\end{theorem}

 In addition, we also give the Ramsey number for delta-systems with an arbitrary finite abelian group coloring, which is the generalized form of Theorem D and stated as follows.

\begin{theorem}\label{Theorem S(r,q,t)} \ Let $G$ be a finite abelian group, and let $r>q\geq 0,m\geq |G|$ be integers with $\exp(G)\mid m$. Then
$$(r-q)m+\max(q,D(G)-1)\leq R(S(r,q,m),G)\leq (r-q)m+\min((r-q)(D(G)-1),|G|-1)+q.$$
\end{theorem}

\section{The Proofs}

We begin this section by remarking that $$s_{k\exp(G)}(G)\geq k\exp(G)+D(G)-1$$ for any $k>0$, which holds by the following extremal example $(\underbrace{0,\ldots,0}\limits_{k\exp(G)-1},g_1,g_2,\ldots,g_{D(G)-1})$ containing no zero-sum subsequence of length $k\exp(G)$, where $g_1,g_2,\ldots,g_{D(G)-1}$ is a zero-sum free sequence of elements in $G$ of length $D(G)-1$.

In \cite{Gao03} Gao introduced the invariant $\ell(G)$ for any $G$, which is defined as
the smallest positive integer $t$ such that $s_{k\exp(G)}(G)-k\exp(G)=D(G)-1$ for every $k\geq t$.
Moreover, he showed in the same paper that
\begin{equation}\label{equation ell(G)}
\frac{D(G)}{\exp(G)}\leq \ell(G)\leq \frac{|G|}{\exp(G)}.
\end{equation}

The following lemma due to Baranyai in 1975 will be crucial in our argument.

\begin{lemma}\label{Lemma Baranyai} (Baranyai) \cite{Bergebook} \ \ Let $n,r$ be integers, $n\geq r\geq 2$, and let $m_1,m_2,\ldots,m_t$ be integers with $m_1+m_2+\cdots+m_t={n\choose r}$. Then $K_n^{(r)}$ is the edge-disjoint sum of $t$ hypergraphs $\mathcal{H}_j$, each satisfying $$m(\mathcal{H}_j)=m_j$$ and $$\left\lfloor\frac{rm_j}{n}\right\rfloor \leq d_{\mathcal{H}_j}(v)\leq \left\lceil\frac{rm_j}{n}\right\rceil$$ for any $v\in V$.
\end{lemma}

\bigskip

Now we are in a position to prove Theorem \ref{Theorem Main}.

\medskip

\noindent {\bf Proof of Theorem \ref{Theorem Main}.}

Notice that $\mathcal{S}_{k\exp(G)}^{(r)}$ is a subset of $\mathcal{I}_{k\exp(G)}^{(r)}$.
Hence, we have
\begin{equation}\label{equation I leq S}
R(\mathcal{I}_{k\exp(G)}^{(r)}, \  G)\leq R (\mathcal{S}_{k\exp(G)}^{(r)}, \  G).
\end{equation}

Denote $n=\Omega(s_{k\exp(G)}(G)).$ We first show that
\begin{equation}\label{equation upper bound for S}
R (\mathcal{S}_{k\exp(G)}^{(r)}, \  G)\leq n.
\end{equation} Let $c:E(K_n^{(r)})\rightarrow G$ be an arbitrary $G$-coloring of the edges of $K_n^{(r)}$. Fix a vertex $v$ in $K_n^{(r)}$. Since the hyperstar $K_n^{(r)}(v)$ has size  $|K_n^{(r)}(v)|={n-1 \choose r-1}\geq s_{k\exp(G)}(G)$, i.e., $(c(e))_{e\in E(K_n^{(r)}(v))}$
is a sequence of elements in $G$ of length $|K_n^{(r)}(v)|\geq s_{k\exp(G)}(G)$, we derive that $K_n^{(r)}(v)$ contains a partial hypergraph,
denoted $\mathcal{H}$, with $m(\mathcal{H})=k\exp(G)$ and $\sum\limits_{e\in E(\mathcal{H})}c(e)=0$. Since $v\in \bigcap\limits_{e\in E(\mathcal{H})} e$, we have that $\mathcal{H}$ is a zero-sum copy of some hyperstar in $\mathcal{S}_{k\exp(G)}^{(r)}$, which proves \eqref{equation upper bound for S}.

Now we show that
\begin{equation}\label{equation I geq n-1}
R (\mathcal{I}_{k\exp(G)}^{(r)}, \  G)\geq n-1.
\end{equation} Take a complete $r$-uniform hypergraph $K_{n-2}^{(r)}$ of order $n-2$. Let $$t=\left\lceil{n-2\choose r} \left\lfloor\frac{n-2}{r} \right\rfloor^{-1}\right\rceil.$$ Applying Lemma \ref{Lemma Baranyai} with $m_1=m_2=\cdots=m_{t-1}=\left\lfloor\frac{n-2}{r} \right\rfloor$ and $m_t={n-2\choose r}-(t-1)\left\lfloor\frac{n-2}{r} \right\rfloor$, we conclude that $K_{n-2}^{(r)}$ is the sum of $t$ edge-disjoint hypermatchings, denoted $\mathcal{M}_1,\mathcal{M}_2,\ldots,\mathcal{M}_t$ respectively, where $|\mathcal{M}_i|=m_i$ for each $i\in \{1,2,\ldots,t\}$. By the virtue of minimality of $n$, we have that $$\begin{array}{llll}
& &{n-2\choose r} \left\lfloor\frac{n-2}{r} \right\rfloor^{-1}\\
&\leq& {n-2\choose r}\frac{r}{n-r-1}\\
&=&{n-2\choose r-1}\\
&<&s_{k\exp(G)}(G),\\
\end{array}$$
and hence $t=\left\lceil{n-2\choose r} \left\lfloor\frac{n-2}{r} \right\rfloor^{-1}\right\rceil<s_{k\exp(G)}(G)$. Take an arbitrary sequence $T=(g_1,g_2,\ldots,g_t)$ of elements in $G$ of length $t<s_{k\exp(G)}(G)$ such that $T$ contains no zero-sum subsequence of length $k\exp(G)$. Then we color the edges of Hypermatching $\mathcal{M}_i$ by the element $g_i$, where $i\in \{1,2,\ldots,t\}$. It is easy to verify that $K_{n-2}^{(r)}$ contains no zero-sum copy of any intersecting family in $\mathcal{I}_{k\exp(G)}^{(r)}$, which proves \eqref{equation I geq n-1}.

It remains to consider the case that $r$ divides $n-1$. Take a complete $r$-uniform hypergraph $K_{n-1}^{(r)}$ of order $n-1$. Then $\left\lceil{n-1\choose r} \left\lfloor\frac{n-1}{r} \right\rfloor^{-1}\right\rceil=\left\lceil{n-1\choose r} \frac{r}{n-1} \right\rceil=
{n-2\choose r-1}<s_{k\exp(G)}(G)$. By a similar argument as above, we have a coloring $c: E(K_{n-1}^{(r)})\rightarrow G$ such that $K_{n-1}^{(r)}$ contains no zero-sum copy of any intersecting family in $\mathcal{I}_{k\exp(G)}^{(r)}$,
which implies that $R (\mathcal{I}_{k\exp(G)}^{(r)}, \  G)>n-1$. Combined with \eqref{equation I leq S} and \eqref{equation upper bound for S}, we have that $R (\mathcal{I}_{k\exp(G)}^{(r)}, \  G)=R (\mathcal{S}_{k\exp(G)}^{(r)}, \  G)=n$ for the case when $r$ divides $n-1$. This completes the proof of the theorem. \qed

\noindent {\bf Remark.} \ It is easy to observe that for the case of
$r>\frac{\Omega(s_{k\exp(G)}(G))-1}{2}$,
$$\begin{array}{llll}R (\mathcal{I}_{k\exp(G)}^{(r)}, \  G)=\left\{\begin{array}{llll}
               \Omega(s_{k\exp(G)}(G))-1,  & \mbox{if \ \ } {\Omega(s_{k\exp(G)}(G))-1 \choose r}\geq s_{k\exp(G)}(G);\\
              \Omega(s_{k\exp(G)}(G)),  & \mbox{if \ \ } otherwise.\\
              \end{array}
              \right.
\end{array}$$

Notice that for the case of $r=2$, i.e., in the graph setting, we have the following corollary of Theorem \ref{Theorem Main}.

\begin{cor} For any integer $k>0$, $$s_{k\exp(G)}(G) \leq R(K_{1,k\exp(G)},G)\leq s_{k\exp(G)}(G)+1.$$
Moreover, if $s_{k\exp(G)}(G)$ is even then  $R(K_{1,k\exp(G)},G)=s_{k\exp(G)}(G)+1.$
\end{cor}

\bigskip

In the rest of this section, we shall prove Theorem \ref{Theorem S(r,q,t)} and
split the proof into lemmas.

 \begin{lemma}\label{Lemma sm(G)GaoYang}
 (\cite{Gaothesis}, \cite{GaoYang}, [\cite{Geroldinger06}, Theorem 5.7.4]) \ For any $G$, $$s(G)\leq |G| + \exp(G) -1.$$
 \end{lemma}

\begin{lemma}\label{Lemma Matching} \ Let $r>0, m\geq |G|$ be integers with $\exp(G)\mid m$. Then
$$R(S(r,0,m),G)\leq  rm+\min(r(D(G)-1),|G|-1).$$
\end{lemma}

\begin{proof} Denote $n=rm+\min(r(D(G)-1),|G|-1)$. Let $c:E(K_{n}^{(r)})\rightarrow G$ be an arbitrary $G-$coloring. It suffices to show that there exists  in $K_{n}^{(r)}$ a zero-sum copy of some delta-system of type $S(r,0,m)$. If $r(D(G)-1)\leq |G|-1$, i.e., $n=rm+r(D(G)-1)=r(m+D(G)-1)$, the conclusion follows from \eqref{equation ell(G)} and the fact that $m=k\exp(G)$ for some $k\geq \ell(G)$ implies any sequence of elements in $G$ of length $m+D(G)-1$ contains a zero-sum subsequence of length exactly $m$. Hence, we may assume that $r(D(G)-1)> |G|-1$ and  $$n=rm+|G|-1.$$

For the case that $m$ is a multiple of $|G|$, the conclusion holds by a similar argument used  by Caro in \cite{Caro91}, which is omitted here. Now we consider the case that $m$ is not a multiple of $|G|$.
Applying Lemma \ref{Lemma sm(G)GaoYang} repeatedly, we can find a hypermatching of size $m-|G|$ in $K_{n}^{(r)}$, say $e_1,e_2,\ldots,e_{m-|G|}$,  with \begin{equation}\label{equation first hypermatching}
c(e_1)+c(e_2)+\cdots+c(e_{m-|G|})=0.
 \end{equation}
 Let \begin{equation}\label{equation A}
 A=V(K_{n}^{(r)})\setminus \bigcup_{i=1}^{m-|G|}e_i.
   \end{equation}
  Since $|A|=|V(H)|-|\bigcup\limits_{i=1}^{m-|G|}e_i|=n-r(m-|G|)=(rm+|G|-1)-r(m-|G|)=r|G|+|G|-1$, it follows that there exists a hypermatching
$e_{m-|G|+1},e_{m-|G|+2},\ldots,e_{m}$ of size $|G|$
  with  \begin{equation}\label{equaiton ej in A}
  e_j\subseteq A
    \end{equation} for all $j\in \{m-|G|+1,m-|G|+2,\ldots,m\}$ and \begin{equation}\label{equation second hypermatching}
    c(e_{m-|G|+1})+c(e_{m-|G|+2})+\cdots+c(e_{m})=0.
       \end{equation}
       By \eqref{equation first hypermatching}, \eqref{equation A},
       \eqref{equaiton ej in A} and \eqref{equation second hypermatching}, we derive that  $e_1,e_2,\ldots,e_m$ is  a zero-sum delta-system of type $S(r,0,m)$ in $K_{n}^{(r)}$. This completes the proof. \end{proof}

       \medskip

\begin{lemma}\label{Lemma affine} \ Let $r>q\geq 0,m\geq |G|$ be integers with $\exp(G)\mid m$. Then $$R(S(r,q,m),G)\leq \min\limits_{0\leq p\leq q}\{p+R(S(r-p,q-p,m),G)\}.$$
\end{lemma}

\begin{proof} If $q=0$, the conclusion means nothing.  Now assume $q>0$. Take an arbitrary integer $0\leq p\leq q$.  Let $$n=R(S(r-p,q-p,m),G).$$ It suffices to show that for an arbitrary $G$-coloring $c:E(K_{n+p}^{(r)})\rightarrow G$ there exists  in $K_{n+p}^{(r)}$ a zero-sum copy of some delta-system of type $S(r,q,m)$. Let $V=V(K_{n+p}^{(r)})$. Take a set $A\subseteq V$ with
\begin{equation}\label{equation |A|=p}
|A|=p.
\end{equation} Let $\mathcal{H}$ be a complete $(r-p)$-uniform hypergraph on the vertex set $V\setminus A$. Let $c':E(\mathcal{H})\rightarrow G$ be a $G$-coloring given by
\begin{equation}\label{equation c'=c}
c'(e')=c(e'\cup A)
\end{equation} for every $e'\in E(\mathcal{H})$. Since $|V(\mathcal{H})|=|V|-|A|=n$, it follows that there exists in $\mathcal{H}$ a zero-sum delta-system of type $S(r-p,q-p,m)$, say $e_1^{'},e_2^{'},\ldots,e_m^{'}$. Let
\begin{equation}\label{equation ei=ei'cup A}
e_i=e_i^{'}\cup A
\end{equation}
 for each $i\in \{1,2,\ldots,m\}$. It follows from \eqref{equation |A|=p}, \eqref{equation c'=c} and \eqref{equation ei=ei'cup A} that $e_1,e_2,\ldots,e_m$ is a zero-sum delta-system of type $S(r,q,m)$ in $K_{n+p}^{(r)}$. This completes the proof.
\end{proof}

\medskip

\begin{lemma}\label{Lemma lower bound} \ Let $r>q\geq 0, m\geq |G|$ be integers with $\exp(G)\mid m$. Then $$R(S(r,q,m),G)\geq (r-q)m+\max(q,D(G)-1).$$
\end{lemma}

\begin{proof} The inequality $R(S(r,q,t),G)\geq (r-q)m+q$ is trivial. Now we  prove  $R(S(r,q,m),G)\geq (r-q)m+D(G)-1$ by giving a $G$-coloring $c:E(K_n^{(r)})\rightarrow G$ such that there exists no zero-sum copy of any delta-system of type $S(r,q,m)$ in $K_n^{(r)}$, where $n=(r-q)m+D(G)-2$.  Let $$V(K_n^{(r)})=V_1\cup V_2$$ be a partition with  $|V_1|=D(G)-1$ and $|V_2|=(r-q)m-1$. Say  $$V_1=\{v_1,v_2,\ldots,v_{D(G)-1}\}$$
and $$V_2=\{v_{D(G)},v_{D(G)+1},\ldots,v_{n}\}.$$ Take a zero-sum free sequence $(g_1,g_2,\ldots,g_{D(G)-1})$ of elements in $G$ of length $D(G)-1$. Let ${\rm f}:V\rightarrow G$ be a map given by
$$\begin{array}{llll}f(v_i)=\left\{\begin{array}{llll}
               g_i,  & \mbox{if \ \ } 1\leq i\leq D(G)-1;\\
               0_G,  & \mbox{if \ \ } D(G)\leq i\leq n.\\
              \end{array}
              \right.
\end{array}$$
Then we define the $G$-coloring $c:E(K_n^{(r)})\rightarrow G$ given by $$c(e)=\sum\limits_{v\in e}f(v)$$ for any edge $e\in E(K_n^{(r)}).$ We can verify that $c$ is the desired $G$-coloring. This completes the proof.
\end{proof}

Therefore, by applying Lemma \ref{Lemma Matching}, Lemma \ref{Lemma affine} with $p=q$, and Lemma \ref{Lemma lower bound}, we have Theorem \ref{Theorem S(r,q,t)} proved.

\section{Concluding remarks}

We remark first that the techniques and arguments of Theorem \ref{Theorem S(r,q,t)} is similar as ones used by Caro \cite{Caro91}, in which Caro also mentioned that ``the essence of his arguments can be generalized quite directly to any finite abelian group.'' However, he did not put out the generalized form even in his later joint paper with Provstgaard (see \cite{CaroProvstgaard99}). For the sake of completeness, in this paper we include Theorem \ref{Theorem S(r,q,t)}, which should belong to Caro.

It is noteworthy that the fact that zero-sum Ramsey number for intersecting family is almost the same as  the zero-sum Ramsey number for hyperstars family seems to have a connection with Erd\H{o}s-Ko-Rado Theorem, which states that in any complete $r$-uniform hypergraph $K_{n}^{(r)}$ of order $n$ with $r\leq \frac{n}{2}$,  the number of edges of a maximum intersecting family is ${n-1\choose r-1}$, exactly the number of edges of the hyperstar $K_{n}^{(r)}(v)$ for any fixed vertex $v\in V(K_n^{(r)})$.
Naturally, combined with Theorem \ref{Theorem Main} we conjecture the following.

\begin{conj} \ Let $k\geq 1$, $r\geq 2$ be positive integers with $r\leq \frac{\Omega(s_{k\exp(G)}(G))-1}{2}$. Then $$ R (\mathcal{I}_{k\exp(G)}^{(r)}, \  G)=R (\mathcal{S}_{k\exp(G)}^{(r)}, \  G).$$
\end{conj}

\bigskip

We close this paper by suggesting a direction for successive researches.
Erd\H{o}s, Chao-Ko and Rado in their original EKR paper (see Theorem 2 of \cite{EKR}) also made a study of the $t$-intersecting family, for which any pair of $r$-uniform edges have an intersection of cardinality at least $t$. The $t$-intersecting family has been studied and generalized in many papers (see \cite{EKR-AhlKha96,EKR-AhlKha,Friedgut08,Tokushige08} for example).
Hence, it would be interesting to study the zero-sum Ramsey number for $t$-intersecting family.

\bigskip

\noindent {\bf Acknowledgements}

\noindent The authors are grateful to Professor Weidong Gao for suggesting this problem and for his helpful discussions.


\begin{thebibliography}{99}

\bibitem{AdhikariRath03} S.D. Adhikari and P. Rath, Remarks on some zero-sum problems, Expo. Math., \textbf{21} (2003) no.2, 185--191.

\bibitem{EKR-AhlKha96} R. Ahlswede and L.H. Khachatrian, The complete nontrivial-intersection theorem for systems of finite sets, J. Combin. Theory Ser. A, \textbf{76} (1996) 121--138.


\bibitem{EKR-AhlKha} R. Ahlswede and L.H. Khachatrian, The complete intersection theorem for systems of finite sets, European J. Combin.,  \textbf{18} (1997) 125--136.

\bibitem{AlonCaro93} N. Alon and Y. Caro, On three zero-sum Ramsey-type problems, J.  Graph Theory, \textbf{17} (1993) 177--192.

\bibitem{Bergebook} C. Berge, Hypergraphs: Combinatorics of Finite Sets, North-Holland Mathematical Library, Amsterdam, 1989.

\bibitem{BialostockiDierker92} A. Bialostocki and P. Dierker, On the Er\H{o}s-Ginzburg-Ziv theorem and the Ramsey numbers for stars and matchings, Discrete Math., \textbf{110} (1992) 1--8.

\bibitem{BialostockiDierker94} A. Bialostocki and P. Dierker, On zero sum Ramsey numbers: Multiple copies of a graph, J.  Graph Theory, \textbf{18} (1994) 143--151.

\bibitem{Caro91} Y. Caro, On zero-sum delta-systems and multiple copies of hypergraphs,  J. Graph Theory, \textbf{15} (1991) 511--521.


\bibitem{Carostar92} Y. Caro, On zero-sum Ramsey numbers--stars, Discrete Math., \textbf{104} (1992) 1--6.


\bibitem{CaroProvstgaard99} Y. Caro and C. Provstgaard, Zero-sum delta-systems and multiple copies of graphs,  J. Graph Theory, \textbf{32} (1999) 207--216.

\bibitem{CaroYuster94} Y. Caro and R. Yuster, A complete characterization of the zero-sum (mod 2) Ramsey numbers, J. Combin. Theory Ser. A, \textbf{68} (1994) 205--211.

\bibitem{Davenport} H. Davenport, Proceedings of the Midwestern conference on group theory and number theory. Ohio State University. April 1966.

\bibitem{EGZ} P. Erd\H{o}s, A. Ginzburg and A. Ziv, Theorem in additive number theory, Bull. Res. Council Israel 10F (1961) 41--43.

\bibitem{EKR} P. Erd\H{o}s, Chao-Ko and R. Rado, Intersecting theorems for systems of finite sets, Quart. J. Math. Oxford, \textbf{12} (1961) 313--318.

\bibitem{Frankl78} P. Frankl, On intersecting families of finite sets,  J. Combin. Theory Ser. A, \textbf{24} (1978) 146--161.

\bibitem{Friedgut08} E. Friedgut, On the measure of intersecting families uniqueness and stability, Combinatorica, \textbf{28} (2008) 503--528.

\bibitem{Gaothesis} W.D. Gao, Some problems in additive group theory and additive number theory, Ph.D. thesis, Sichuan University, 1994.

\bibitem{Gaofundamental96}    W.D. Gao, A combinatorial problem on finite abelian groups, J. Number Theory, \textbf{58} (1996) 100--103.

\bibitem{Gao03} W.D. Gao, On zero-sum subsequences of restricted size -- II, Discrete Math., \textbf{271} (2003) 51--59.

\bibitem{GaoGeroldingersurvey} W.D. Gao and A. Geroldinger, Zero-sum problems in finite abelian groups: a survey, Expo. Math.,  \textbf{24} (2006) 337--369.

\bibitem{GaoHan}  W.D. Gao and D.C. Han, On zero-sum subsequences of length $k \exp(G)$, manuscript.

\bibitem{GaoThangadurai06}    W.D. Gao and R. Thangadurai, On zero-sum sequences of prescribed length, Aequationes Math., \textbf{72} (2006) 201--212.

\bibitem{GaoYang} W.D. Gao and Y.X. Yang, Note on a combinatorial constant,
J. Math. Res. Exposition, \textbf{17}(1997) 139--140, in chinese.


\bibitem{Geroldinger06} A. Geroldinger and F. Halter-Koch, Non-unique factorizations. Algebraic, Combinatorial and Analytic Theory, Pure Appl. Math., vol. 278, Chapman and Hall/CRC, 2006.


 \bibitem{GeroldingerRuzsa} A. Geroldinger and I.Z. Ruzsa,  Combinatorial Number Theory and Additive Group Theory, Advanced Courses in Mathematics CRM Barcelona, Birkh$\ddot{a}$user, 2009, pp. 48--56.

 \bibitem{EKR-HiltonMilner} A.J.W. Hilton and E.C. Milner, Some intersection theorems for systems of finite sets, Quart. J. Math. Oxford, \textbf{18} (1967) 369--384.

\bibitem{Kubertin} S. Kubertin, Zero-sums of length $kq$ in $\mathbb{Z}_q^d$, Acta Arith., \textbf{116} (2005) 145--152.


\bibitem{Talbot03} J. Talbot,  Intersecting Families of Separated Sets, J. London Math. Soc., \textbf{68} (2003) 37--51.

\bibitem{Tokushige08} N. Tokushige, Brace-Daykin type inequalities for intersecting families, European J. Combin.,  \textbf{29} (2008) 273--285.





\end{thebibliography}
\end{document}